\documentclass[11pt,draft,reqno]{amsart}
\usepackage{amsmath,amssymb,amsfonts,amscd}
\usepackage{amscd}
\usepackage{graphics}
\usepackage{dsfont}
\usepackage[all]{xy}
\usepackage{enumerate}
\usepackage{dsfont}
\usepackage{bbold}
\usepackage{mathbbol}
\usepackage{mathrsfs}

\usepackage[all]{xy}
\usepackage{enumerate}

{\catcode`\@=11
\gdef\n@te#1#2{\leavevmode\vadjust{%
 {\setbox\z@\hbox to\z@{\strut#1}%
  \setbox\z@\hbox{\raise\dp\strutbox\box\z@}\ht\z@=\z@\dp\z@=\z@%
  #2\box\z@}}}
\gdef\leftnote#1{\n@te{\hss#1\quad}{}}
\gdef\rightnote#1{\n@te{\quad\kern-\leftskip#1\hss}{\moveright\hsize}}
\gdef\?{\FN@\qumark}
\gdef\qumark{\ifx\next"\DN@"##1"{\leftnote{\rm##1}}\else
 \DN@{\leftnote{\rm??}}\fi{\rm??}\next@}}
\DeclareFontFamily{OT1}{wncyr}{\hyphenchar\font45 }
\DeclareFontShape{OT1}{wncyr}{m}{n}{%
   <5> <6> <7> <8> <9> gen * wncyr
   <10> <10.95> <12> <14.4> <17.28> <20.74>  <24.88>wncyr10}{}
\DeclareFontShape{OT1}{wncyr}{m}{it}{%
   <5> <6> <7> <8> <9> gen * wncyi
   <10> <10.95> <12> <14.4> <17.28> <20.74> <24.88> wncyi10}{}
\DeclareFontShape{OT1}{wncyr}{m}{sc}{%
   <5> <6> <7> <8> <9> <10> <10.95> <12> <14.4>
   <17.28> <20.74> <24.88>wncysc10}{}
\DeclareFontShape{OT1}{wncyr}{b}{n}{%
   <5> <6> <7> <8> <9> gen * wncyb
   <10> <10.95> <12> <14.4> <17.28> <20.74> <24.88>wncyb10}{}
\input cyracc.def
\def\rus{\usefont{OT1}{wncyr}{m}{n}\cyracc\fontsize{9}{11pt}\selectfont}

\DeclareMathSizes{9}{9}{7}{5} 

\theoremstyle{plain}

\newtheorem{theorem}{Theorem}
\newtheorem{lemma}{Lemma}
\newtheorem{proposition}{Proposition}
\newtheorem{corollary}{Corollary}

\newtheorem*{conjsnonumber}{Conjectures}

\theoremstyle{definition}

\theoremstyle{remark}
\newtheorem{remark}{\it Remark}
\newtheorem{example*}{Example}

\def\XT+{{\rm X}(T)_+}
\def\Xt{{\rm X}(T)}

\begin{document}

\

\vskip -8mm

\
\title[Equations defining affine algebraic groups]
{On the equations defining\\ affine algebraic groups}

\author[Vladimir L. Popov]{Vladimir L. Popov}
\address{Steklov Mathematical Institute, Russian Academy of Sciences, Gubkina 8, Moscow 119991, Russia}

\address{National Research University\\ Higher School of Economics, Myas\-nitskaya
20, Moscow 101000,\;Russia}

\email{popovvl@mi.ras.ru}

\thanks{
Supported by
 grant {\rus RFFI}
14-01-00160.\newline
\indent \;2010 {\it Mathematics Subject Classification}: Primary 14L10;
Secondary: 20G05.\;{\it Key\-words and Phrases}: reductive algebraic group, Borel subgroup, weight, module, orbit.}

\begin{abstract}
For the coordinate
algebras of connected affine algebraic groups, we explore the problem of finding a
presentation by generators and relations canonically determined by the group structure. \end{abstract}


\maketitle

\section{Introduction}

Connected algebraic groups constitute a remarkable class of irreducible quasiprojective algebraic varieties.\;It contains the subclasses of abeli\-an varieties and affine algebraic groups.\;These
subclasses are basic:\;by
Chevalley's theorem, every connected algebraic group $G$  has a unique
connected normal affine algebraic subgroup $L$ such that $G/L$ is an abelian variety;
whence
the variety $G$ is an $L$-torsor over the abelian variety $G/L$.\;The varieties from these subclasses can be embedded in many ways as closed subvarieties in, respectively, projective and affine spaces.\;A natural question then arises as to whether there are distinguished embeddings and equations of their images, which are
canonically determined by the group structure.\;For abelian varieties, this is the existence problem for canonically defined bases in linear systems and that of
presentating homogeneous coordinate rings of ample  invertible sheafs by generators and relations.\;These problems were explored and solved by D.\;Mumford  \cite{Ma66}.\;For
affine algebraic groups, it is the existence problem of the canonically defined presentations of the coordinate algebras of such groups by generators and relations.\;We explore this problem in the present paper.

We fix as the base field an algebraically closed field $k$ of arbitrary characteristic.\;In this paper,
as in \cite{Bor91}, ``variety'' means ``algebraic variety'' in the sense of
Serre
\cite[Subsect.\;34]{Se55}; every variety is taken over $k$.

Let $G$ be a connected affine algebraic group and let $R_u(G)$ be its unipotent radical.\;In view of \cite[Props.\,1,\,2]{Gr58}, \cite[Thm.\,10]{Ro56}, the underlying variety of $G$
is isomorphic to the product of that of $G/R_u(G)$ and $R_u(G)$, and the latter is isomorphic to an affine space.\;Therefore,
the problem under consideration is reduced to the case of reductive groups.\;Given this, henceforth $G$ stands for a connected {\it reductive} algebraic group.

The simplest
case of ${\rm SL}_{2}$ is the guiding example.\;Take  the polynomial $k$-algebra $k[x_1,x_2,x_3,x_4]$ in four variables $x_i$.\;The usual presentation of $k[{\rm SL}_{2}]$ is given by the surjective homomorphism
\begin{equation}\label{SL2}
\mu\colon k[x_1, x_2, x_3, x_4]\to  k[{\rm SL}_{2}],\quad\mbox{
$\mu(x_i)\biggl(\!\renewcommand{\arraycolsep}{2pt}\begin{bmatrix}
a_1 & a_2\\
a_3 & a_4
\end{bmatrix}\!\biggr)\!=\!a_i$},
\end{equation}
whose kernel is the ideal $(x_1x_4-x_2x_3-1)$.\;After rewriting, this presentation can be interpreted in
terms of the group structure of ${\rm SL}_2$ as follows.

We have $k[x_1, x_2, x_3, x_4]=
k[x_1, x_3]\otimes_k k[x_2, x_4]$ and the restriction of $\mu$ to the subalgebra
$k[x_1, x_3]$ (respectively, $k[x_2, x_4]$) is an isomorphism with the subalgebra
${\mathcal S}^+$ (respectively, ${\mathcal S}^-$) of $k[{\rm SL}_{2}]$
consisting of all regular functions
invariant with respect to
the subgroup $U^+$ (respectively, $U^-$) of all unipotent upper (respectively, lower) triangular matrices acting by right translations.\;Hence 
\eqref{SL2} yields the following presentation of $k[{\rm SL_2}]$ by generators and relations:
\begin{equation}\label{SL2P}
\begin{split}
k[{\rm SL_2}]&\cong ({\mathcal S}^+\otimes _{k}{\mathcal S}^{-})/{\mathcal I},\\
{\mathcal S}^+&=k[\mu (x_1),\mu (x_3)]\cong k[x_1, x_3],\\
{\mathcal S}^{-}&=
k[\mu(x_2),\mu (x_4)]\cong k[x_2, x_4],\\
{\mathcal I}&=\big(\mu (x_1)\otimes \mu(x_4)-\mu(x_2)\otimes \mu(x_3)-1\big).
\end{split}
\end{equation}\vskip 1mm

The subgroups $U^+$, $U^-$ are opposite maximal unipotent subgroups of ${\rm SL}_2$. The subalgebras ${\mathcal S}^+$, ${\mathcal S}^-$ are stable with respect to ${\rm SL}_2$ acting by left translations, and $f\!:=\!\mu (x_1)\otimes \mu(x_4)\!-\!\mu(x_2)\otimes \mu(x_3)\!-\!1$ is the unique element of $({\mathcal S}^+\otimes_k{\mathcal S}^-)^{{\rm SL}_2}$
determined by the conditions
$f(e,e)\!=\!1$, $k[f]\!=\!({\mathcal S}^+\otimes_k{\mathcal S}^-)^{{\rm SL}_2}$.

We show that there is an analogue of \eqref{SL2P} for every connected reductive
algebraic group $G$.\;Namely, we endow $k[G]$ with the $G$-module structure determined by left translations and fix in $G$ a pair of opposite Borel subgroups $B^+$ and $B^-$.\;Let $U^\pm$ be the unipotent radical of $B^\pm$.\;Consider the $G$-stable subalgebras
\begin{equation}\label{S+}
\begin{split}
{\mathcal S}^+&:=\{ f\in k[G]\mid  f(gu)=f(g)\; \mbox{for all $g\in G, u\in U^+$} \},\\
{\mathcal S}^-&:=\{ f\in k[G]\mid  f(gu)=f(g)\; \mbox{for all $g\in G, u\in U^-$} \}
\end{split}
\end{equation}
of $k[G]$ and the natural multiplication homomorphism of $k$-algebras
\begin{equation}\label{mu}
\mu \colon  {\mathcal S}^+\otimes _{k}{\mathcal S}^{-}\hskip -1mm\to k[G],\quad f_1\otimes f_2\mapsto f_1f_2.
\end{equation}

 For $k=\mathbb C$, the following conjectures were put forward  in \cite{FT92}:

 \begin{conjsnonumber}[D.\,E.\,Flath and J. Towber, 1992]\

 \begin{enumerate}[\hskip 2.2mm\rm(i)]
 \item[\rm (S)] The homomorphism $\mu$ is surjective.
 \item[\rm (K)] The ideal ${\rm ker}\,\mu$ in ${\mathcal S}^+\otimes_k {\mathcal S}^-$
is generated by $({\rm ker}\,\mu)^G$.
 \end{enumerate}
 \end{conjsnonumber}

 If these conjectures are true, then the problem under consideration is reduced to
 the following:

\begin{enumerate}[\hskip 4.2mm \rm(a)]
 \item find the canonically defined generators of  the $k$-algebra $({\rm ker}\,\mu)^G$,
 \item find the canonically defined presentations of ${\mathcal S}^\pm$ by generators and relations.
 \end{enumerate}

 In \cite{FT92}, Conjectures (S) and (K) were proved  for $k=\mathbb C$ and $G={\rm SL}_n$,
 ${\rm GL}_n$, ${\rm SO}_n$, ${\rm Sp}_n$ by means of lengthy
 direct
 computations of
 some Laplace decompositions, minors, and algebraic identities between them.\;In Theorems \ref{sur} and \ref{ker}
 below we prove Conjectures (S) and (K) in full generality, with no restrictions on $k$ and $G$.

In Theorems  \ref{vect}, \ref{iG} below we describe ${\rm ker}\,\mu$ as a vector space over $k$.\;In Theorem \ref{gi}, we
solve the above part (a) of the problem, finding the cano\-ni\-cally
defined generators of the $k$-algebra $({\rm ker}\,\mu)^G$.\,We call them {\it ${\rm SL}_2$-type relations} of the sought-for canonical presentation of $k[G]$ because
for $G={\rm SL}_2$ the element
$\mu (x_1)\otimes \mu(x_4)-\mu(x_2)\otimes \mu(x_3)-1$
is just such a generator of
$\mathcal I$ (see \eqref{SL2P}).\;All
of them are inhomogeneous of degree $2$.\;If $G$ is
semisimple,
  they are indexed by the elements of the Hilbert basis ${\mathscr H}$ of the monoid of dominant weights of $G$.\;Note that the cardinality $|{\mathscr H}|$ of ${\mathscr H}$ is at least ${\rm rank}\,G$ with equality for simply connected $G$, but in the general case it may be much bigger.\;For instance, if $G={\rm PGL}_{r}$, then $|{\mathscr H}|\geqslant p(r)+\varphi(r)-1$, where $p$ and $\varphi$ are, respectively, the classical partition function and the Euler function (see \cite[Example\;3.15]{Po11}).\;Note that
the problem of determining a full set of generators of the ideal ${\rm ker}\,\mu$ was formulated in
\cite[Sect.\,4]{Fl94} and, for $k=\mathbb C$,  $G={\rm SL}_n$, ${\rm GL}_n$, ${\rm SO}_n$, ${\rm Sp}_n$, solved in
\cite{FT92} by lengthy direct computations.

For a semisimple group $G$
whose monoid of dominant weights is freely generared (i.e., with $|{\mathscr H}|={\rm rank}\, G$),
 a solution to the above part (b) of the problem  in characteristic $0$ was
 obtained (but not published) by B.\;Kostant; his proof appeared in \cite[Thm.\;1.1]{LT79}.\;In arbitrary characteristic, such a solution is given by Theorems
 \ref{RR}, \ref{fg},
 \ref{grS} below, which are
heavily based on
the main results of \cite{RR85} and \cite{KR87}.\;All relations in this case are homogeneous
of degree $2$.\;We call them {\it Pl\"ucker-type relations} of the sought-for canonical presentation of $k[G]$ because  the $k$-algebra ${\mathcal S}^\pm$ for $G={\rm SL}_n$ is the coordinate algebra of the affine multicone over the flag variety, and if ${\rm char}\,k=0$, these relations are generated by the classical
Pl\"ucker-type relations, obtained by Hodge \cite{Ho42}, \cite{Ho43}, that determine this multicone
(see below Section \ref{example}).\;The
set of these relations is a union of
finite-dimensional vector spaces canonically determined by the group structure of $G$; these spaces are
indexed by the elements of ${\mathscr H}\times {\mathscr H}$ and
different spaces have zero intersection (see Theorem \ref{grS}).\;Thus
in this case, we obtain a canonical presentation of $k[G]$, in which all relations are quadratic and divided into two families: homogeneous relations
of Pl\"ucker type and inhomogeneous relations
 of ${\rm SL}_2$-type.\;As a parallel, we
recall that any abelian variety is canonically presented as an intersection of quadrics in a projective space given by the Riemann equations; see \cite{Ke98}, \cite{LB92}.

For an arbitrary reductive group $G$, let $\tau\colon \widehat G\to G$ be the universal covering.\;Then $\widehat G=Z\times C$, where $Z$ is a torus, $C$ is a simply connected semisimple group, $G=\widetilde G/{\rm ker}\,\tau$, and ${\rm ker}\,\tau$ is a finite central subgroup.\;The algebra ${\mathcal S}^\pm$ for $\widehat G$ is then
the tensor product of $k[Z]$ and the algebra ${\mathcal S}^\pm$ for $C$.\;Since the presentation of
$k[Z]$ is clear, and that of ${\mathcal S}^\pm$ for $C$ are given by Theorems \ref{RR}, \ref{fg}, \ref{grS}, the above part (b) of the problem is reduced to finding a presentation for the invariant algebra of the finite abelian group ${\rm ker}\,\tau$.

As an illustration, in the last Section \ref{example} we consider the example of $G={\rm SL}_n$, ${\rm char}\,k=0$ and
describe explicitly how the ingredients of our construction
and the canonical presentation of $k[G]$ look like in this case.

\vskip 3mm

The preprints \cite{Po95}, \cite{Po00} of these results in characteristic $0$ have been disseminated long ago.\;The
validity of the results
in arbitrary characteristic was announced in \cite{Po00}.\;The author is pleased by the arisen occasion to finally present
the complete proofs.

\vskip 4mm

\noindent{\it Notation and conventions.}

\vskip 1mm

Below we use freely the standard notation and conventions of
\cite{Bor91}, \cite{Ja87}, \cite{PV94}, and \cite{Sh13}.\;In
parti\-cu\-lar, the algebra of functions regular on a variety $X$ is denoted by $k[X]$, the field of rational functions on an irreducible $X$ is denoted by $k(X)$, and the local ring of $X$ at a point $x$ is denoted by ${\mathcal O}_{x, X}$.\;For a morphism $\varphi\colon X\to Y$ of varieties, $\varphi^*\colon k[Y]\to k[X]$ denotes its co\-mor\-phism.

All topological terms refer to the Zariski topology; the closure of $Z$ in $X$
is denoted by $\overline Z$ (each time it is clear from the context  what is $X$).

  The fixed point set of an action of a group $P$ on a set $S$ is denoted by\;$S^P$.\;Every action $\alpha\colon H\times X\to X$ of an algebraic group $H$ on a variety $X$ is always assumed to be regular (the latter means that $\alpha$ is  a morphism).\;For every $h\in H$, $x\in X$, we write $g\cdot x$ in place of $\alpha(g, x)$.\;The $H$-orbit and the $H$-stabilizer of $x$ are denoted respectively by $H\cdot x$ and $H_x$.\;Every homomorphism of algebraic groups is assumed to be algebraic.

  The additively written group of characters (i.e., homomorphisms to the multiplicative group of $k$) of an algebraic group $H$ is denoted by ${\rm X}(H)$.\;The value of a character $\lambda\in {\rm X}(H)$ at an element $h\in H$ is denoted by $h^\lambda$.\;Given
  a $kH$-module $M$, its weight space with weight $\lambda\in X(H)$ is denoted by $M_\lambda$.

  We fix in $G$ the maximal torus
  \begin{equation*}
  T:=B^+\cap B^-
  \end{equation*}
 and identify ${\rm X}(B^\pm)$ with ${\rm X}(T)$ by means of the restriction isomorphisms ${\rm X}(B^\pm)\to {\rm X}(T)$, $\lambda\mapsto \lambda|_{T}^{\ }$.

 By $\XT+$ we denote
the monoid of dominant weights of $T$ determined by $B^+$.\;Below the highest weight of every simple $G$-module is assumed to be
the highest weight with respect to $T$ and $B^+$.

 We denote by $w_0$ be the longest element of the Weyl group of $T$ and fix in the normalizer of $T$ a representative ${\overset{.}{w}}_0$ of $w_0$.\;We then have ${\overset{.}{w}}_0B^\pm{\overset{.}{w}}_0^{-1}=B^\mp$ and
${\overset{.}{w}}_0U^\pm{\overset{.}{w}}_0^{-1}=U^\mp$.\;For every $\lambda \in \XT+$, we put $\lambda^*:=-w_0(\lambda)\in \XT+$.

The set of all nonnegative rational numbers is denoted by ${\mathbb Q}_{\geqslant 0}$ and we put
${\mathbb N}:={\mathbb Z}\cap {\mathbb Q}_{\geqslant 0}$.

If $m\in \mathbb Z$, $m>0$, we put $[m]:=\{a\in \mathbb Z\mid 1\leqslant a\leqslant m\}$.

 For $d\in \mathbb N$, we
 denote by $[m]_d$ the set of all increasing sequences
 of $d$ elements of $[m]$ (if $d\notin [m]$, then $[m]_d=\varnothing$).



\section{Proof of Conjecture (S)}

For every $\lambda \in \Xt$, the spaces
\begin{equation}\label{S+la}
\begin{split}
{\mathcal S}^+(\lambda)&:=\{ f\in {\mathcal S}^+\mid f(gt)=t^{\lambda }f(g)\; \mbox{for all $g\in G, t\in T$} \},\\
{\mathcal S}^-(\lambda)&:=\{ f\in {\mathcal S}^-\mid f(gt)=t^{w_0(\lambda) }f(g)\; \mbox{for all $g\in G, t\in T$} \}.
\end{split}
\end{equation}
are the finite-dimensional (see, e.g.,\;\cite[I.5.12.c)]{Ja87}) $G$-submodules of the $G$-modules ${\mathcal S}^+$ and ${\mathcal S}^-$ respectively.\;Since ${\mathcal S}^-(\lambda)$ is the right translation of ${\mathcal S}^+(\lambda)$ by ${\overset{.}{w}}_0,$ these $G$-submodules are isomorphic.\;In the notation of \cite[II.2.2]{Ja87}, we have
\begin{equation}\label{H}
{\mathcal S}^-(\lambda)=H^0(\lambda^*),
\end{equation}
so by \eqref{H} and \cite[II.2.6, 2.2, 2.3]{Ja87}, the following properties hold:
\begin{equation}\label{Sm}
\left.
\begin{split}
&{\hskip 1mm\rm (i)}\;{\mathcal S}^\pm(\lambda)\neq 0 \iff \lambda\in \XT+;\\[-1.2mm]
&{\rm (ii)}\;\mbox{${\rm soc}_G^{\ }{\mathcal S}^\pm(\lambda)$
is a simple $G$-module with the highest weight $\lambda^*$.}
\end{split}
\right\}
\end{equation}

If ${\rm char}\, k=0$, then the $G$-module ${\mathcal S}^+(\lambda)$ is semisimple and hence ${\mathcal S}^+(\lambda)={\rm soc}_G^{\ }{\mathcal S}^+(\lambda)$ by \eqref{Sm}(ii).\;If ${\rm char}\, k>0$, then, in general,  this equality does not hold.
From \eqref{S+}, \eqref{S+la}, and \eqref{Sm}(i) we infer that
\begin{equation}\label{grS+-}
\begin{split}
{\mathcal S}^+&=\displaystyle \bigoplus_{\lambda \in \XT+}
{\mathcal S}^+(\lambda), \quad {\mathcal S}^+(\lambda){\mathcal S}^+(\mu)\subseteq {\mathcal S}^+(\lambda+\mu),\\
{\mathcal S}^-&=\displaystyle \bigoplus_{\lambda \in \XT+}
{\mathcal S}^-(\lambda), \quad {\mathcal S}^-(\lambda){\mathcal S}^-(\mu)\subseteq {\mathcal S}^-(\lambda+\mu),
\end{split}
\end{equation}
i.e., \eqref{grS+-} are the $\XT+$-gradings of the algebras ${\mathcal S}^+$ and ${\mathcal S}^-$.\;They are obtained from each other by the right translation by ${\overset{.}{w}}_0$.

\begin{theorem}\label{RR} The
linear span of ${\mathcal S}^\pm(\lambda){\mathcal S}^\pm(\mu)$ over $k$ is ${\mathcal S}^\pm(\lambda+\mu)$.
\end{theorem}
\begin{proof}
This statement is the main result of \cite{RR85}.\;Note that the difficulty lies in the case of positive characteristic: since ${\mathcal S}^\pm$ is an integral domain, if
${\rm char}\,k=0$, then the claim immediately follows from \eqref{Sm}(i) and
the inclusions in \eqref{grS+-} because then ${\mathcal S}^\pm(\lambda+\mu)$ is a simple
$G$-module.
\end{proof}

\begin{theorem}\label{fg}\

\begin{enumerate}[\hskip 2.2mm\rm(i)]
\item If ${\mathscr G}$ is a generating set of the semigroup $\XT+$, then the $k$-algebra ${\mathcal S}^\pm$ is generated
by the subspace $\bigoplus_{\lambda\in {\mathscr G}}{\mathcal S}^\pm(\lambda)$.
\item The $k$-algebras $\mathcal {\mathcal S}^+$ and ${\mathcal S}^-$ are finitely generated.
\end{enumerate}
\end{theorem}
\begin{proof} Part (i) follows from \eqref{grS+-} and
Theorem \ref{RR}.\;Being the intersection of the lattice $\Xt$ with a convex cone in
$\Xt\otimes_{\mathbb Z} \mathbb Q$ generated by finitely many vectors,
the semigroup $\XT+$ is finitely ge\-ne\-ra\-ted.\;This, (i), and the inequality
$\dim_k{\mathcal S}^\pm(\lambda)<\infty$ imply (ii).
\end{proof}

Now we are ready to turn to the proof of Conjecture (S).

\begin{theorem}\label{sur}
The homomorphism
$\mu$ is surjective.
\end{theorem}

Our proof of Theorem \ref{sur} is based on two general results.\;The first is the following
well-known
surjectivity criterion:
\begin{lemma}\label{ce} The following properties of a morphism $\varphi\colon X\to Y$
of affine algebraic varieties are equivalent:
\begin{enumerate}[\hskip 4.2mm \rm(a)]
 \item
 $\varphi$ is a closed embedding;
 \item
$\varphi^*\colon k[Y]\to k[X]$ is surjective.
\end{enumerate}
\end{lemma}
\begin{proof}
See, e.g.,\;\cite[1.5]{St74}.
\end{proof}

The second  is the closedness criterion for
orbits of
connected solvable affine algebraic groups
that generalizes
Rosenlicht's classical theorem on closedness of orbits of unipotent groups
\cite[Thm.\,2]{Ro61}.

\begin{theorem}\label{clos} Let a connected solvable affine algebraic group $S$
act on an affine algebraic variety $Z$.\;Let $x$ be a point
of $Z$.\;Consider the orbit morphism $\tau: S\rightarrow Z$, $s\mapsto s\cdot z$.\;Then
the following
properties
are equivalent:
\begin{enumerate}[\hskip 4.2mm\rm(a)]
\item the orbit $S\cdot z$ is closed in $Z$;
\item the semigroup
$\{\lambda \!\in\! {\rm X}(S)\mid \mbox{
\hskip -1mm the function
$S\!\to\! k$, $s\!\mapsto\! s^{\lambda}$ lies in $\tau^{*}(k[Z])$}\!\}$
is a group.
\end{enumerate}
\end{theorem}
\begin{proof} This is proved in \cite[Thm.\,4]{Po88}
\end{proof}

\begin{remark}\label{group} Since ${\rm X}(S)$ in Theorem \ref{clos} is a finitely generated free abelian group, it can be naturally regarded as a lattice in
 ${\rm X}(S)\otimes_{\mathbb Z}\mathbb Q$.\;Hence
the following general criterion is applicable
for verifying condition (b).

Let $M$ be a nonempty subset of a finite
dimensional vector space $V$ over $\mathbb Q$.\;Let ${\mathbb Q}_{\geqslant 0}M$,
${\rm conv} M$, and $\mathbb QM$ be, respectively, the convex cone generated by $M$,
 the convex hull of $M$, and  the linear span of $M$ in $V$.\;Then
the
following properties are equivalent (see \cite[p.\,386]{Po88}):
\begin{enumerate}[\hskip 2.2mm \rm(i)]
\item $0$ is an interior point of ${\rm conv} M$,
\item $\mathbb Q_{\geqslant 0} M=\mathbb Q M$.
\end{enumerate}
If $M$ is a subsemigroup of $V$, then
(i) and (ii) are equivalent to
\begin{enumerate}[\hskip -6.0mm \rm(iii)]
\item $M$ is a group.
\end{enumerate}
\end{remark}

\begin{proof}[Proof of Theorem {\rm \ref{sur}}]\

1. We consider the action of $G$ on its underlying algebraic variety by left
translations.\;By Theorem \ref{fg}, there is
an irreducible affine algebraic variety $X$ endowed with an action
of $G$ and a $G$-equivariant dominant morphism
\begin{equation}\label{alpha}
\alpha \colon G\to X\;\;\mbox{such that $\alpha^*$ is an isomorpshism $k[X]\xrightarrow{\cong}{\mathcal S}^+$.}
\end{equation}

Let $x:=\alpha(e)$.\;Since $\alpha$ is $G$-equivariant, we have
\begin{equation}\label{ag}
\alpha(g)=g\cdot x \;\;\mbox{for every $g\in G$},
\end{equation}
 and since $\alpha$ is dominant,
the orbit $G\cdot x$
is open and dense in $X$.\;Consider the canonical projection
$\pi \colon G\to G/U^+$.\;It is the geometric quotient for the action of $U^+$ on $G$ by right translations.\;Therefore, \eqref{S+} yields the isomorphism
\begin{equation}\label{pi*}
\pi^*\colon k[G/U^+]\xrightarrow{\cong}{\mathcal S}^+
\end{equation}
and,  since $\alpha$ is constant on the fibers $\pi$,
there exists a $G$-equivariant
morphism $\iota\colon G/U^+\to X$ such that
\begin{equation}\label{aip}
\alpha=\iota\circ\pi.
\end{equation}

From \eqref{aip} we infer that the image of $\iota$ is $G\cdot x$.\;Since the group $U^+$ is unipotent, the algebraic variety $G/U^+$ is quasiaffine
(see \cite[Thm.\,3]{Ro61}). There\-fore, $k(G/U^+)$ is the field of fractions
of $k[G/U^+]$.\;On the other hand, $k(X)$ is the field of fractions of $k[X]$ inasmuch as $X$ is affine.\;Using that \eqref{aip} and isomorphisms \eqref{alpha}, \eqref{pi*} yield the isomorphism
$\iota^*\colon k[X]\xrightarrow{\cong} k[G/U^+],$
we conclude that $\iota $ is a birational isomorphism.\;Therefore, for
a point $z$ in ge\-ne\-ral position in $G\cdot x$ the fiber $\iota^{-1}(z)$ is a single point.\;Being $G$-equivariant, $\iota $ is then injective.\;Finally, since $G$ is smooth, $k[G]$ is integrally closed; therefore, ${\mathcal S}^\pm$ is integrally closed as well in view of \eqref{S+} (see, e.g., \cite[Thm.\,3.16]{PV94}).\;Thus $X$ is normal, and hence by \;Zariski's Main Theorem, $\iota\colon G/U^+\to G\cdot x$ is an isomorphism.\;Using that $\pi$ is separable (see, e.g., \cite[II.6.5]{Bor91}), from this we infer that the following properties hold:
    \begin{enumerate}[\hskip 3.2mm\rm (i)]
      \item[$({\rm i}_1)$] $G_x=U^+$;
      \item[$({\rm ii}_1)$] $G\to G\cdot x$, $g\mapsto \alpha(g)=g\cdot x$, is a separable morphism.
    \end{enumerate}

\vskip 1mm

2. Let $y:={\overset{.}{w}}_0\cdot x$.\;Consider the $G$-equivariant morphism
\begin{equation}\label{beta}
    \beta\colon G\to X,\;\;g\mapsto g\cdot y.
\end{equation}
From \eqref{ag}, \eqref{beta}, $({\rm i}_1)$, and $({\rm ii}_1)$ we infer that the following properties hold:
 \begin{enumerate}[\hskip 4.8mm\rm (i)]
      \item[$({\rm i}_2)$] $G_y=U^-$;
      \item[$({\rm ii}_2)$] $G\to G\cdot y$, $g\mapsto \beta(g)=g\cdot y$, is a separable morphism;
      \item[$({\rm iii}_2)$] $\beta^*$ is an isomorphism  $k[X]\xrightarrow{\cong} {\mathcal S}^-$.
    \end{enumerate}

    \vskip 1mm

3. Now consider the $G$-equivariant  morphism
\begin{equation}\label{gamma}
    \gamma:=\alpha\times\beta\colon G\to X\times X,\quad g\mapsto g\cdot z,\;\;\mbox{where $z:=(x, y)$.}
\end{equation}
From \eqref{gamma} and $({\rm i}_1)$, $({\rm i}_2)$ we obtain
\begin{equation}\label{trivial}
G_z=G_x\cap G_y
=U^+\cap U^-
=\{e\},
\end{equation}
hence $\gamma$ is injective.
We claim that $\gamma$ is a closed embedding, i.e.,
\begin{enumerate}[\hskip 4.2mm\rm(a)]
  \item $G\to G\cdot z$, $g\mapsto g\cdot z$ is an isomorphism;
  \item $G \cdot z$ is closed in $X\times X$.
\end{enumerate}

If this claim is proved, then the proof of Theorem \ref{sur}
is completed as follows.\;Consider the isomorphism
\begin{equation}\label{theta}
   k[X]\otimes_kk[X]\to k[X\times X],\quad f\otimes h\mapsto fh.
\end{equation}
Then \eqref{mu}, \eqref{alpha}, $({\rm iii}_3)$, \eqref{gamma}, \eqref{theta} imply that $\mu$ is the composition of the homomorphisms
\begin{equation}\label{ep}
    {\mathcal S}^+\otimes_k{\mathcal S}^-\xrightarrow[\cong]{(\alpha^*)^{-1}\otimes (\beta^*)^{-1}}
k[X]\otimes_kk[X]
\xrightarrow[\cong]{\eqref{theta}}
k[X\times X]\xrightarrow{\gamma^*}k[G].
\end{equation}
Hence the surjectivity of $\mu$ is equivalent to the surjectivity of $\gamma^*$.\;By Lemma
\ref{ce},
the latter is equivalent to the property that $\gamma$ is a closed embedding, i.e., that properties (a) and (b) hold.

Thus the proof of Theorem \ref{sur} is reduced to proving properties (a) and (b).

\vskip 1mm

4. First, we shall prove property (a).\;Since $\gamma$ is injective, this is reduced to proving separability of $\gamma$.\;In turn, in view of \eqref{gamma}, the latter is reduced to proving that ${\rm ker}(d\gamma)_e$ is contained in ${\rm Lie}\,G_z$, i.e., that ${\rm ker}(d\gamma)_e=\{0\}$
because of \eqref{trivial} (see \cite[II.6.7]{Bor91}).\;Using
loc.cit., from \eqref{ag}, \eqref{beta},  $({\rm i}_1)$, $({\rm ii}_1)$,  $({\rm i}_2)$, $({\rm ii}_2)$ we infer that ${\rm ker}(d\alpha)_e\subseteq {\rm Lie}\,U^+$, ${\rm ker}(d\beta)_e\subseteq {\rm Lie}\,U^-$.\;In view of \eqref{gamma}, we then have
${\rm ker}(d\gamma)_e\!=\!{\rm ker}(d\alpha)_e\cap {\rm ker}(d\beta)_e\!\subseteq\!
{\rm Lie}\,U^+\cap {\rm Lie}\,U^-=\{0\}$.\;This proves property (a).

\vskip 1mm

5. Now we shall prove property (b).\;Actually, we shall prove the stronger property that the orbit $B^+\cdot z$ is closed in $X\times X$: since the algebraic variety $G/B^+$ is complete, this stronger property implies property (b) (see \cite[Sect.\,2.13, Lemma\,2]{St74}).\;Using that $B^+$ is connected solvable, to this end we shall apply Theorem \ref{clos}.

Namely, consider the morphism $\tau\colon B^+\to X\times X$, $b\mapsto b\cdot z$ and
the following subsemi\-group $M$ in ${\rm X}(B^+)$:
\begin{equation*}\label{M}
    M:=\{\lambda \!\in\! {\rm X}(B^+)\mid \mbox{
    the function $B^+\to k$, $b\mapsto b^{\lambda}$ lies in $\tau^{*}(k[X\times X])$.}\}
\end{equation*}
We identify ${\rm X}(B^+)$ with the lattice in $L:={\rm X}(B^+)\otimes _{\mathbb Z}\mathbb Q$.\;In view of Theorem \ref{clos} and Remark \ref{group}, the orbit $B^+\cdot z$ is closed if and only if
\begin{equation}\label{=}
\mathbb Q_{\geqslant 0} M=\mathbb Q M.
\end{equation}
Given this, the problem is reduced to proving that property \eqref{=} holds.\;This is done below.

\vskip 1mm

6. Since  $\tau=\gamma|^{\ }_{B^+}$, the algebra $\tau^*(k[X\times X])$ is the image of the homomorphism $\gamma^*(k[X\times X])\to k[B^+]$, $f\mapsto f|^{\ }_{B^+}$.\;From \eqref{ep} we see that
$\gamma^*(k[X\times X])$ contains ${\mathcal S}^+$ and ${\mathcal S}^-$.\;Hence the restrictions of
${\mathcal S}^+$ and ${\mathcal S}^+$ to $B^+$ lie in $\tau^*(k[X\times X])$.\;We
shall exhibit some characters of $B^+$ lying in these restrictions.

First consider the restriction of ${\mathcal S}^+(\lambda)$ to $B^+$ for $\lambda\in \XT+$.\;Note that ${\mathcal S}^+(\lambda)$ contains a function $f$ such that $f(e)\neq 0$.\;Indeed, in view of \eqref{Sm}(i) and Borel's fixed point theorem, ${\mathcal S}^+(\lambda)$ contains a $B^-$-stable line $\ell$.\;The group $B^-$ acts on $\ell$ by means of
a character $\nu\in {\rm X}(B^-)$.\;Take a nonzero function $f\in \ell$.\;For every $b\in B^-$, $u\in U^+$, we then have $f(b^{-1}u)=b^\nu f(u)\overset{\eqref{S+}}{=}b^\nu f(e)$; whence $f(e)\neq 0$ because
$B^-U^+$ is dense in $G$.\;This proves the existence of $f$.\;Multiplying $f$ by $1/f(e)$, we may assume
that $f(e)=1$.\;Then  for every $b\in B^+$, we deduce from $\eqref{S+}$, $\eqref{S+la}$ that
$f(b)=b^\lambda f(e)=b^{\lambda}$, i.e., $f|^{\ }_{B^+}$ is the character $B^+\!\to k$, $b\mapsto b^\lambda$.
This proves the inclusion
 \begin{equation}\label{M+}
{\rm X}(B^+)^{\ }_+\subseteq M.
 \end{equation}

Now consider the restriction of ${\mathcal S}^-(\lambda)$ to $B^+$ for $\lambda\in \XT+$.\;In view of \eqref{Sm}(ii), there is a $B^+$-stable line $\ell$  in ${\mathcal S}^-(\lambda)$, on which $B^+$ acts by the character $\lambda^*\in {\rm X}(B^+)$.\;Take a nonzero function $f\in \ell$.\;We may assume that
$f(e)=1$: this is proved as above with $\nu=\lambda$, replacing $B^-$ by $B^+$, and $U^+$ by
$U^-$.\;For every $b\in B^+$, we then have $f(b^{-1})=b^{\lambda^*}$, i.e., $f|^{\ }_{B^+}$ is the character $B^+\!\to k$, $b\mapsto b^{-\lambda^*}=b^{w_0(\lambda)}$.\;This proves the inclusion
 \begin{equation}\label{M-}
-{\rm X}(B^+)^{\ }_+\subseteq M.
 \end{equation}

 Since ${\mathbb Q}^{\ }_{\geqslant 0}({\rm X}(B^+)^{\ }_+)-{\mathbb Q}^{\ }_{\geqslant 0}({\rm X}(B^+)^{\ }_+)=L$, the inclusions \eqref{M+}, \eqref{M-} imply the equality ${\mathbb Q}^{\ }_{\geqslant 0}M=L$; whence a fortiori the equality \eqref{=} holds.
This completes the proof of Theorem \ref{sur}.
\end{proof}

\section{Proof of Conjecture (K)}

We now intend to describe  the ideal ${\rm ker}\,\mu$ in ${\mathcal S}^+\otimes_k {\mathcal S}^-$.\;This is done in Sections 3 and 4 in several steps: first in Theorem \ref{ker} we prove that ${\rm ker}\,\mu$
is generated by $({\rm ker}\,\mu)^G$, then in Theorem \ref{vect} we describe ${\rm ker}\,\mu$ as a vector space, and finally in Theorem \ref{gi} we find a standard finite generating set of ${\rm ker}\,\mu$.

The first step is based on the following general statement:
\begin{theorem}\label{Luna}
Let $Z$ be an affine algebraic variety endowed with an action of a reductive algebraic group $H$.\;Let $a\in Z$ be a point such that the orbit morphism
\begin{equation*}
\varphi\colon H\to Z,\;\;h\mapsto h\cdot a
\end{equation*}
 is a closed embedding.\;Then the ideal ${\rm ker}\,\varphi^*$ in $k[Z]$ is generated by
 $({\rm ker}\,\varphi^*)^H$.
\end{theorem}

For the proof of Theorem \ref{Luna} we need the following

\begin{lemma}\label{idealm}
Let $\psi\colon Y\to Z$ be a morphism of irreducible affine algebraic varieties and let $z\in \psi(Y)$ be a smooth point of
$Z$.\;Assume that for each point $y\in \psi^{-1}(z)$
the following hold{\rm:}
\begin{enumerate}[\hskip 2.2mm\rm(i)]
\item $y$ is a smooth point of $\,Y$;
\item
the differential
$d_y\psi$ is surjective.
\end{enumerate}
 Then the ideal
$\{f\in k[Y]\mid f|_{\psi^{-1}(z)}=0\}$
of $k[Y]$ is generated by $\psi^*(\mathfrak m)$, where
$\mathfrak m:=\{h\in k[Z]\mid h(z)=0\}$.
\end{lemma}

\begin{proof}
Given  a nonzero function $f\in k[Y]$, below we denote by $Y_f$  the principal open subset $\{y\in Y\mid f(y)\neq 0\}$ of $Y$;\;it is affine and  $k[Y_f]=k[Y]_f$.

\vskip 1mm

1.\;Let $s_1,\ldots, s_d$ be a system of generators of the ideal $\mathfrak m$ of $k[Z]$.\;Put $t_i:=\psi^*(s_i)$.\;Then we have
\begin{equation}\label{fifi}
\{y\in Y\mid t_1(y)=\cdots=t_d(y)=0\}=\psi^{-1}(z).
\end{equation}

We claim  that, for every point $a\in Y$, there is a function $h_a\in k[Y]$ such that the principal open subset $U=Y_{h_a}$ is a neighborhood of $a$ and  $I^{\ }_U:=\{f\in k[U]\mid
f|_{\psi^{-1}(z)\cap U}=0\}$ is the ideal of $k[U]$
generated by $t_1|_U,\ldots, t_d|_U$.

Proving this, we consider two cases.

First,
 consider the case where $a\notin \psi^{-1}(z)$.\;Then any principal open neighborhood of $a$ not intersecting
$\psi^{-1}(z)$ may be taken as $U$ because in this case $I^{\ }_U=k[U]$ and, in view of \eqref{fifi} and Hilbert's Nullstellensatz,
$k[U]=
k[U]t_1|_U+\cdots+ k[U]t_d|_U$.

Second,
 consider the case where $a\in \psi^{-1}(z)$.\;Let $n=\dim Y, m=\dim Z$. Since $a$ and $z$ are the smooth points, the assumption (ii) yields the equality
\begin{equation}\label{n-m}
\dim {\rm ker}\,d_a\psi=n-m.
\end{equation}
The functions $s_1,\ldots, s_d$ generate the maximal ideal of ${\mathcal O}_{z, Z}$.\;Therefore,
re\-num\-bering them if necessary, we may (and shall) assume that $s_1,\ldots, s_m$ is a system of local parameters of $Z$ at $z$, i.e., $\bigcap_{i=1}^m{\rm ker}\,d_zs_i=\{0\}$.\;Since $d_at_i=d_a\psi\circ d_zs_i$, we then infer from (ii)  that $\bigcap_{i=1}^m{\rm ker}\,d_at_i={\rm ker}\,d_a\psi$.\;In view of \eqref{n-m}, the latter equality implies the existence of functions $f_1,\ldots, f_{n-m}\in {\mathcal O}_{a, Y}$ such that
$t_1,\ldots, t_m, f_1,\ldots, f_{n-m}$ is a system of local parameters of $Y$ at $a$.\;Let
\begin{equation}\label{part}
F:=\{y\in Y\mid t_1(y)=\cdots=t_m(y)=0\}.
\end{equation}
By \cite[Chap.\,II, \S3, Sect.\,2, Thm.\,4]{Sh13}, there is a principal open neighborhood $U$ of $a$ such that
$F\cap U$ is an irreducible smooth $(n-m)$-dimensional closed subvariety of $U$ whose ideal in $k[U]$ is generated by $t_1|_U,\ldots, t_m|_U$.\;On the other hand, \eqref{fifi} and \eqref{part} yield $\psi^{-1}(z)\subseteq F$ and, by the fiber dimension theorem,  every irreducible component of
$\psi^{-1}(z)$ has dimension $\geqslant n-m$.\;Hence $U\cap F=\psi^{-1}(z)\cap U$.\;This and \eqref{fifi}
prove the claim.

\vskip 1mm

2.\;Using this claim, the proof of Lemma \ref{idealm} is completed as follows. Since $Y=\bigcup_{a\in Y} Y_{h_a}$ and $Y$ is quasi-compact, there exists a finite set of points $a_1,\ldots, a_r\in Y$ such that
\begin{equation}\label{covering}
Y=\displaystyle  \bigcup_{i=1}^r Y_{h_{i}},\;\;\mbox{where $h_i:=h_{a_i}$.}
\end{equation}

Now, let $f\in k[Y]$ be a function such that $f|_{\psi^{-1}(z)}=0$.\;Then, in view of the definition of $h_a$, for every $i=1,\ldots, r$,  we have
\begin{equation}\label{decomps}
fh_i^{b_i}=c_{i,1}t_1+\cdots +c_{i,d}t_d\quad\mbox{for some $c_{i,j}\in k[Y]$ and $b_i\in{\mathbb N}$}.
\end{equation}
From \eqref{covering} and Hilbert's Nullstellensatz we infer that there are functions $q_1,\ldots, q_r\in k[Y]$
such that
\begin{equation}\label{pe}
1=q_1h_1^{b_1}+\cdots + q_sh_s^{b_r}.
\end{equation}
From \eqref{decomps} and \eqref{pe} we then deduce that
\begin{equation*}
f=\displaystyle \bigg(\sum_{i=1}^rq_ic_{i,1}\bigg)t_1+\cdots+\bigg(\sum_{i=1}^rq_ic_{i,d}\bigg)t_d\in
k[Y]t_1+\cdots+k[Y]t_d.
\end{equation*}
This completes the proof.
\end{proof}

\begin{proof}[Proof of Theorem {\rm \ref{Luna}}]
There is a closed equivariant embedding of $Z$ in an affine space on which $H$ operates linearly (see \cite[Lemma 2]{Ro61}, \cite[Thm.\,1.5]{PV94}).\;Hence we may (and shall) assume that $Z$ is an irreducible smooth affine algebraic variety.

Since $G$ is reductive, $k[Z]^G$ is a finitely generated $k$-algebra
(see, e.g., \cite[Thm.\,A.1.0]{MF82} and the references therein).\;Denote by
 $Z/\!\!/H$ the affine algebraic variety ${\rm Specm}(k[Z]^G)$ and by
 $\pi
 \colon Z\to Z/\!\!/H$ the morphism corresponding to the inclusion homomorphism
 $k[Z]^G\hookrightarrow k[Z]$.

The condition on the point $a$ implies that its $H$-stabilizer is trivial,
\begin{equation}\label{Ha}
H_a=\{e\}.
\end{equation}
Hence $H\cdot a$ is a closed  $H$-orbit of maximal dimension.\;Taking into account that in
 every fiber of $\pi$ there is a unique closed
 orbit lying in the closure of every orbit contained in this fiber (see
 \cite[Cors.\,1.2, A.1.0]{MF82}), from this we deduce the equality
\begin{equation}\label{of}
\pi^{-1 }
(\pi
(a))=H\cdot a.
\end{equation}

Since  the group $\{e\}$ is linearly reductive, from
\eqref{Ha} and
the separability of $\varphi$ we infer
by \cite[Prop.\,7.6]{BR85}  that there is  a smooth affine subvariety $S$ of the $H$-variety $Z$, which is an \'etale slice at $a\in S$.\;In view of \eqref{Ha}, this means the following:
\begin{enumerate}[\hskip 2.2mm\rm(i)]
\item $\pi|^{\ }_S\colon S\to Z/\!\!/H$ and $\psi\colon H\times S\to Z$, $(h, s)\mapsto h\cdot s$  are the    \'etale
morphisms;
\item the diagram
\begin{equation*}\label{cd}
\begin{matrix}
\xymatrix{ H\times S\ar[r]^{\hskip 3mm \psi}\ar[d]_{{\rm pr}^{\ }_2}& Z\ar[d]^{\pi}\\
S\ar[r]^{\pi|^{\ }_S}&Z/\!\!/H
}
\end{matrix}
\end{equation*}
is a Cartesian square, i.e., it is commutative and the morphism $H\times S \to S\times_{Z/\!\!/H} Z$ determined by $\psi$ and ${\rm pr}_2$ is an isomorphism.
\end{enumerate}

From (i) and (ii) we deduce  that $\pi(a)$ is a smooth point of $Z/\!\!/H$ and the differentials $d_{(e,a)}\psi$, $d_a(\pi|^{\ }_{S})$ are isomorphisms.\;Since $d_{(e,a)}{{\rm pr}^{\ }_2}$ is clearly surjective, (ii)
then implies that
$d_a\pi$ is surjective, too.

Now, in view of \eqref{of} and transitivity of the action of $H$ on $H\cdot a$, we conclude that $d_z\pi$ is surjective for every point $z\in \pi^{-1}(\pi(a))$.\;In view of Lemma \ref{idealm}, this implies the claim of Theorem \ref{Luna}.
\end{proof}

\begin{theorem}\label{ker}
The ideal ${\rm ker}\,\mu$ in ${\mathcal S}^+\otimes_k {\mathcal S}^-$
is generated by $({\rm ker}\,\mu)^G$.
\end{theorem}

\begin{proof} In the proof of Theorem \ref{sur} we have shown that
\begin{enumerate}[\hskip 2.2mm ---]
\item the homomorphism $\mu$ is the composition of the homomorphisms \eqref{ep};
\item the morphism $\gamma$
is a closed
embedding.
\end{enumerate}
In view of these facts, Theorem \ref{ker} is equivalent to the claim that the ideal
${\rm ker}\,\gamma^*$ in $k[X\times X]$ is generated by $({\rm ker}\,\gamma^*)^G$.\;This
claim
follows from Theorem\;\ref{Luna}.
\end{proof}

\section{Structure of $({\rm ker}\,\mu)^G$}

We shall use the following lemma for describing $({\rm ker}\,\mu)^G$ as a vector space.
\begin{lemma}\label{otimes}
\begin{align}
\dim({\mathcal S}^+(\lambda)\otimes_k {\mathcal S}^-(\nu)\!)^G&=
\Big\{\hskip -1.5mm\begin{array}{ll}
1& \mbox{if $\;\nu=\lambda^*$,}\\
0& \mbox{if $\;\nu\neq\lambda^*$}
\end{array}
\;\;
\mbox{for every $\lambda, \nu\in \XT+$},\label{10}\\
({\mathcal S}^+\otimes_k{\mathcal S}^-)^G&=\displaystyle \bigoplus_{\lambda\in \XT+}({\mathcal S}^+(\lambda)\otimes_k {\mathcal S}^-(\lambda^*)\!)^G.\label{G=}
\end{align}
\end{lemma}
\begin{proof}  In view of \eqref{grS+-}, the equality \eqref{G=} follows from \eqref{10}.\;To prove \eqref{10}, we note that
\begin{equation*}
({\mathcal S}^+(\lambda)\otimes_k {\mathcal S}^-(\nu)\!)^G\cong {\rm Hom}_G
({\mathcal S}^+(\lambda)^*, {\mathcal S}^-(\nu))
\end{equation*}
and, in view of \eqref{H}, the $G$-module
${\mathcal S}^+(\lambda)^*$ is
the universal highest weight module of weight $\lambda$
(the Weyl mo\-du\-le);
in particular,
for each $G$-module $M$, there is an isomorphism
\begin{equation}\label{we}
{\rm Hom}_G({\mathcal S}^+(\lambda)^*, M))\xrightarrow{\cong} (M^{U^+})_\lambda,
\end{equation}
where the right-hand side of \eqref{we} is the weight space of $T$  (see \cite[II.2.13, Lemma]{Ja87}).\;Sin\-ce ${\mathcal S}^-(\nu)^{U^+}$\;is a line
on which $B^+$ acts by means of $\nu^*$ (see \cite[II.2.2, Prop.]{Ja87}), this proves \eqref{10}.
\end{proof}

We identify $k[G]\otimes_k k[G]$ with $k[G\times G]$ by the isomorphism
\begin{equation}\label{ide}
k[G]\otimes_k k[G]\to k[G\times G],\quad
f_1\otimes f_2
\mapsto
((a,b)
\mapsto f_1(a)f_2(b)).
\end{equation}
Thus ${\mathcal S}^+\otimes_k {\mathcal S}^-$ is regarded as a subalgebra of $k[G\times G]$,
and \eqref{mu}, \eqref{ide} yield the equality
\begin{equation}\label{rst}
f(a,a)=\mu(f)(a)\;\;\mbox{for every $f\in {\mathcal S}^+\otimes_k {\mathcal S}^-$ and $a\in G$.}
\end{equation}

\begin{theorem}\label{vect}\

\begin{enumerate}[\hskip 2.2mm \rm (i)]
  \item If $f\in ({\mathcal S}^+\otimes_k {\mathcal S}^-)^G$, then
  $f-f(e,e)\in ({\rm ker}\,\mu)^G$.
  \item Every $h\in  ({\rm ker}\,\mu)^G$\;can be uniquely written in the form
  \begin{equation}\label{decom}
    h=\displaystyle \sum
    (h_{\lambda}-h_{\lambda}(e,e)),\quad\mbox{
    $h_\lambda\in ({\mathcal S}^+(\lambda)\otimes_k {\mathcal S}^-(\lambda^*)\!)^G$,}
  \end{equation}
the sum is taken over a finite set of nonzero elements $\lambda\in {\rm X}\hskip -.2mm (T)\hskip -.2mm_+\hskip -.2mm.$
  \end{enumerate}
\end{theorem}
\begin{proof}
(i) Since $\mu$
is $G$-equivariant, its restriction to $({\mathcal S}^+\otimes_k{\mathcal S}^-)^G$ is a homomorphism to $k[G]^G=k$.\;Hence $\mu(f)$ is
a constant.\;In view of \eqref{rst}, this implies (i).

(ii) If \eqref{decom} holds, then the decomposition \eqref{G=} implies that $h_\lambda$ is the natural projection of
$h$ to $({\mathcal S}^+(\lambda)\otimes_k {\mathcal S}^-(\lambda^*)\!)^G$ determined by this decomposition; whence the uniqueness of \eqref{decom}.\;To prove the existence, let $h_\lambda$ be
the aforementioned projection of $h$ to $({\mathcal S}^+(\lambda)\otimes_k {\mathcal S}^-(\lambda^*)\!)^G$.\;Then
$h=\sum_{\lambda\in F} h_\lambda$ for a finite set $F\subset\XT+$.\;Hence
$0=\mu(h)=\sum_{\lambda\in F}\mu(h_\lambda)$.\;As above, $\mu(h_\lambda)=h_{\lambda}(e,e)$; this implies the equality \eqref{decom}, where the sum is taken over all $\lambda\in F$.\;Since $h_0$ is a constant, we may assume that $F$ does not contain $0$.\;This proves\;(ii).
\end{proof}

In
the next lemma, for brevity,  we put (cf.\;\cite{Ja87})
\begin{equation}\label{notatio}
\begin{gathered}
V(\lambda):={\mathcal S}^-(\lambda)^*\cong {\mathcal S}^+(\lambda)^*,\quad  L(\lambda):=V(\lambda)/{\rm rad}_GV(\lambda),\\[-.5mm]
\pi_\lambda\colon V(\lambda)\to L(\lambda)\;\;\mbox{is the canonical projection.}
\end{gathered}
\end{equation}
The $G$-module $V(\lambda)$ (hence $L(\lambda)$ as well) is generated by a $B^+$-stable line of weight $\lambda$ (see \cite[II, Sect.\;2.13, Lemma]{Ja87}); whence $V(\lambda)$ is also ge\-ne\-rated by a $B^-$-stable line of weight $-\lambda^*$.

Also, for the $G$-modules $P$ and $Q$, we denote by
${\mathscr B}(P\times Q)$
the $G$-module of all bilinear maps $P\times Q\to k$;\;we then have
the isomorphism of $G$-modules
\begin{equation}\label{bd}
P^*\otimes_k Q^*\xrightarrow{\cong} {\mathscr B}(P\!\times\!Q),\quad f\otimes h\mapsto fh.
\end{equation}

\begin{lemma}\label{bili}For all elements $\lambda, \nu\in \XT+$, the following hold{\rm:}
\begin{enumerate}[\hskip 4.2mm\rm(a)]
  \item $\dim {\mathscr B}(V(\lambda)\!\times\!V(\nu)\!)^G\hskip -1mm =\Big\{\hskip -1.5mm\begin{array}{ll}
1& \mbox{if $\;\nu=\lambda^*\!,$}\\
0& \mbox{if $\;\nu\neq\lambda^*\!.$}
\end{array}$
  \item $\dim {\mathscr B}(L(\lambda)\!\times\!L(\nu)\!)^G\hskip -1mm =\Big\{\hskip -1.5mm\begin{array}{ll}
1& \mbox{if $\;\nu=\lambda^*\!,$}\\
0& \mbox{if $\;\nu\neq\lambda^*\!.$}
\end{array}$
  \item  Every nonzero element $\theta\in {\mathscr B}(L(\lambda)\times
  L(\lambda^*)\!)^G$ is a nondegenerate pair\-ing $
      L(\lambda)\!\times\!L(\lambda^*)\to k${\rm.}
  \item If
  $\,l^+\in L(\lambda)$, $l^-\in L(\lambda^*)$ are the nonzero semi-invariants of, res\-pec\-ti\-vely, $B^+$ and $B^-$, then $\theta(l^+, l^-)\neq 0$ for $\theta$ from {\rm (c)}.\;For every nonzero ele\-ment $\epsilon\in k$, there exists a unique $\theta$ such that $\theta(l^+, l^-)=\epsilon$.
  \item Every element $\vartheta\in {\mathscr B}(V(\lambda)\!\times\!V(\lambda^*)\!)^G$\;vanishes on ${\rm ker}\,\pi^{\ }_{\lambda}\times {\rm ker}\,\pi^{\ }_{{\lambda}^*}$.\;If $\vartheta\neq 0$, then $\vartheta$ is a nondegenerate pairing $V(\lambda)\times V(\lambda^*)\to k$.
  \item Let $v^+\in V(\lambda)$
  and $v^-\in V(\lambda^*)$ be, respectively,
  the non\-zero $B^+$\!\!\!-
and $B^-$\!\!\!-semi-invari\-ants of weights $\lambda$ and $-\lambda$
  that generate the $G$-modules $V(\lambda)$ and $V(\lambda^*)$.\;Then
  $\vartheta(v^+, v^-)\neq 0$ for every nonzero
  element
  $\vartheta\in {\mathscr B}
  (V(\lambda)\times V(\lambda^*)\!)^G$.
\end{enumerate}
\end{lemma}
\begin{proof} 
Part (a) follows from \eqref{10}, \eqref{bd}, \eqref{notatio}.\;Part (b) is proved similarly, using that $L(\lambda)$ is a simple $G$-module with highest weight $\lambda$ (see \cite[II.2.4]{Ja87}).\;The simplicity of $L(\lambda)$  implies (c) because the left and right kernels of $\theta$ are $G$-stable.

Proving (d), take a basis $\{p_1,\ldots, p_s\}$ of $L(\lambda)$ such that
$p_1=l^+$
and eve\-ry  $p_i$ is a weight vector of $T$.\;Let $\{p_1^*,\ldots, p_s^*\}$ be the basis of $L(\lambda^*)$ dual to $\{p_1,\ldots, p_s\}$ with respect to $\theta$.\;Let
$L(\lambda)'$ be
the linear span over $k$ of all $p_i$'s with $i>1$.\;Then $L(\lambda)'$ is $B^-$-stable, and, for every element $u\in U^-$, we have $u\cdot p_1=
p_1+p'$, where $p'\in L(\lambda)'$ (see, e.g.,\;\cite[Sect.\;3.3, Prop.\;2 and p.\,84]{St74}).\;Then, for every elements $\alpha_1,\ldots,\alpha_s\in k$, we have
\begin{align*}
(u\cdot p^*_1)\bigg(\displaystyle \sum_{i=1}^s\alpha_ip_i\bigg)&=
p^*_1\bigg(\displaystyle \sum_{i=1}^s\alpha_i(u^{-1}\cdot p_i)\!\bigg)\\
&=p^*_1(\alpha_1p_1+\mbox{an element of $L(\lambda)'$})\\
&=\alpha_1=p^*_1\bigg(\displaystyle \sum_{i=1}^s\alpha_ip_i\!\bigg);
\end{align*}
whence $u\cdot p^*_1\!=\!p^*_1$.\;Therefore,  $l^-\!=\!\lambda p_1^*$ for a nonzero $\lambda\in k$, hence
$\theta(l^+, l^-)=\lambda\neq 0$.\;This and (b) prove (d).

It follows from \eqref{notatio}, (a), (b) that the embedding
\begin{equation*}
    {\mathscr B}(L(\lambda)\!\times\!L(\lambda^*))^G\to
    {\mathscr B}(V(\lambda)\!\times\!V(\lambda^*))^G,\quad \theta\mapsto \theta\circ (\pi_{\lambda}\times \pi_{\lambda^*})
\end{equation*}
is an isomorphism.\;Part (e) follows from this and (c).

Part (f) follows from (d) and (e), because $\pi_\lambda(v^+)$ and $\pi_{\lambda^*}(v^-)$ are, in view of \eqref{notatio}, the
nonzero semi-invariants of, respectively, $B^+$ and $B^-$.
\end{proof}

\begin{lemma}\label{covv}
Let an algebraic group $H$ act on a algebraic variety $Z$ and let $V$ be a finite-dimensional
submodule of the $H$-module $k[Z]$.\;Then the morphism
\begin{equation}\label{cova}
    \varphi\colon Z\to V^*,\quad \mbox{$\varphi(a)(f)=f(a)$ for every $a\in Z, f\in V$}
\end{equation}
has the following properties:
\begin{enumerate}[\hskip 2.2mm\rm(i)]
  \item $\varphi$ is $H$-equivariant;
  \item the restriction of $\;\varphi^*$ to $(V^*)^*$ is
  an isomorphism $(V^*)^*\to V$;
  \item $\varphi^*$ exercises an isomorphism between $k[\overline{\varphi(Z)}]$ and the subalgebra of $k[Z]$ generated by $V$.
\end{enumerate}
\end{lemma}
\begin{proof} Part (i) is proved by direct verification.

Every function $f\in V$
determines an element $l_f\in (V^*)^*$ by the formula $l_f(s)=s(f)$, $s\in V^*$.\;It is immediate that $V\to (V^*)^*$, $f\mapsto l_f$ is a vector space isomorphism and that \eqref{cova} implies $\varphi^*(\l_f)=f$.\;This proves (ii).

Let $\iota\colon (V^*)^*\to  k[\overline{\varphi(Z)}]$ be the restriction homomorphism.\;The $k$-algebra $k[\overline{\varphi(Z)}]$ is generated by $\iota((V^*)^*)$.\;Part (iii) now follows from the fact that
$\varphi^*$ exercises an embedding of $k[\overline{\varphi(Z)}]$ in $k[Z]$ and, in view of (ii), the image of $\iota((V^*)^*)$
under this embedding is $V$.
\end{proof}

\begin{corollary}\label{coro1} In the notation of Lemma {\rm \ref{covv}},
let $V\neq \{0\}$ and let the orbit $H\cdot a$ be dense in $Z$.\;Then  $\varphi(a)\neq 0$.
\end{corollary}

We call the morphism  \eqref{cova} the {\it covariant determined by the submodule $V$}.

\begin{lemma}\label{phipm} Let $\lambda$ be an element of $\XT+$ and let
\begin{equation*}
\varphi^+\colon G\to {\mathcal S}^+(\lambda)^*, \quad
\varphi^-\colon G\to {\mathcal S}^-(\lambda^*)^*
\end{equation*}
be the covariants determined by the submodules ${\mathcal S}^+(\lambda)$ and ${\mathcal S}^-(\lambda^*)$ of
the $G$-module
$k[G]$.\;Then $v^+:=\varphi^+(e)$ and $v^-:=\varphi^-(e)$ are, respectively, the nonzero $B^+$\!\!- and $B^-$\!\!-se\-mi\-invariants of weights $\lambda$ and $-\lambda$.
\end{lemma}
\begin{proof}
First, we have
$v^+\neq 0$, $v^-\neq 0$ by Corollary \ref{coro1}.\;Next, for every  $f\in {\mathcal S}^+(\lambda)$, $b\in B^+$, we have
\begin{align*}(b\cdot v^+)(f)&=\varphi^+(e)(b^{-1}\cdot f)\overset{\eqref{cova}}{=\hskip -1mm=}
(b^{-1}\cdot f)(e)\\[-1.5mm]
&=f(b)\overset{\eqref{S+la}}{=}b^{\lambda}f(e)\overset{\eqref{cova}}{=\hskip -1mm=}(b^{\lambda}v^+)(f);\end{align*}
whence $b\cdot v^+=b^{\lambda}v^+$, i.e., $v^+$ is a nonzero $B^+$-semi-invariant of weight $\lambda$, as claimed.\;For $v^-$ the proof is similar.
\end{proof}
\begin{theorem}\label{iiss}
The restriction of $\mu$ to $({\mathcal S}^+(\lambda)\otimes_k {\mathcal S}^-(\lambda)\!)^G$
for every $\lambda\in \XT+$
is an isomorphism $({\mathcal S}^+(\lambda)\otimes_k {\mathcal S}^-(\lambda^*)\!)^G\xrightarrow{\cong} k[G]^G=k$.
\end{theorem}
\begin{proof}
In view of
\eqref{rst}
and Lemma \ref{otimes}, the proof is reduced to showing that there is a function $f\in
({\mathcal S}^+(\lambda)\otimes_k {\mathcal S}^-(\lambda^*)\!)^G$ such that $f(e,e)\neq 0$.

Consider the covariants
$\varphi^+$ and $\varphi^-$ from Lemma \ref{phipm} and the $G$-equivariant morphism
\begin{equation*}
 \varphi:=   \varphi^+\times \varphi^-\colon G\times G\to {\mathcal S}^+(\lambda)^*\times {\mathcal S}^-(\lambda^*)^*.
\end{equation*}
Lemma \ref{bili}(a) and \eqref{notatio} imply that ${\mathscr B}({\mathcal S}^+(\lambda)^*\times {\mathcal S}^-(\lambda^*)^*)^G$ contains a nonzero element $\vartheta$.\;By Lemma \ref{covv}, the function
$f:=\vartheta\circ\varphi\colon G\times G\to k$ is contained  in $({\mathcal S}^+(\lambda)\otimes_k {\mathcal S}^-(\lambda^*)\!)^G$.\;For this $f$, using Lemmas \ref{phipm} and \ref{bili}(f), we obtain
\begin{equation}\label{varth}
f(e,e)=\vartheta(\varphi(e,e))=\vartheta(\varphi^+(e), \varphi^-(e))\neq 0.
\end{equation}
This completes the proof.
\end{proof}

\begin{corollary}\label{corrr}
For every element $\lambda\in\XT+$, there exists a unique element
\begin{equation}\label{s}
s_\lambda\in ({\mathcal S}^+(\lambda)\otimes_k {\mathcal S}^-(\lambda^*)\!)^G\subseteq k[G\times G] \;\;\mbox{such that
$s_{\lambda}(e,e)=1$.}
\end{equation}
If $\{f_1,\ldots, f_d\}$ and $\{h_1,\ldots , h_d\}$ are the bases of ${\mathcal S}^+(\lambda)$ and ${\mathcal S}^-(\lambda^*)$ dual with respect to a nondegenerate $G$-invariant pairing
${\mathcal S}^+(\lambda)\times {\mathcal S}^-(\lambda^*)\to k$ {\rm(}the latter exists by {\rm \eqref{bd}} and Lemma {\rm\ref{bili})}, then $\varepsilon:=\sum_{i=1}^d f_i(e)h_i(e)\neq 0$ and
\begin{equation*}
s_\lambda=\varepsilon^{-1}\bigg(\displaystyle \sum_{i=1}^d f_i\otimes h_i\bigg).
\end{equation*}
\end{corollary}
\begin{proof} First, note that if $P$, $Q$ are the finite dimensional $kG$-modules,  $\theta\in {\mathscr B}(P, Q)^G$ is a nondegenerate pairing $P\times Q\to k$, and
$\{p_1,\ldots, p_m\}$ and $\{q_1,\ldots, q_m\}$ are the bases of $P$ and $Q$ dual with respect to $\theta$, then $\sum_{i=1}^m p_i\otimes q_i$ is a nonzero element of $(P\otimes_k Q)^G$ (not depending on the choice of these bases). Indeed, $\theta$ determines the isomorphism
of $G$-modules
\begin{equation}\label{mmmm}
\begin{split}
    \phi\colon
    P\otimes_k Q&\to {\rm Hom}(P, P),\\
    (\phi(p\otimes q))(p')&=\theta(p', q)p,\;\mbox{where $p, p'\in P,\, q\in Q$}.
    \end{split}
\end{equation}
From \eqref{mmmm} we then obtain
\begin{align*}
\bigg(\phi\bigg(\displaystyle  \sum_{i=1}^mp_i\otimes q_i\bigg)\!\!\bigg)(p_j)&=
\displaystyle \sum_{i=1}^m\theta(p_j, q_i)p_i=\displaystyle \sum_{i=1}^m\delta_{ij} p_i=p_j,
\end{align*}
therefore, $\phi\big(\sum_{i=1}^mp_i\otimes q_i\big)={\rm id}_{P}$; whence the claim.

For $P={\mathcal S}^+(\lambda)$, $Q={\mathcal S}^-(\lambda^*)$, it yields that $\sum_{i=1}^d f_i\otimes h_i$ is a nonzero element of $({\mathcal S}^+(\lambda)\otimes_k {\mathcal S}^-(\lambda^*))^G$.\;Theorem \ref{iiss} and \eqref{rst} then complete the proof.
\end{proof}

\begin{remark}\label{000} For ${\rm char}\,k=0$, there is another characterization of $s_\lambda$.\;Namely, let
${\mathscr U}$ be the universal enveloping algebra of ${\rm Lie}\,G$.\;Every
${\mathcal S}^\pm(\lambda)$ is endowed with the natural ${\mathscr U}$-module structure.\;Let
$\{x_1,\ldots, x_n\}$ and $\{x_1^*,\ldots, x_n^*\}$ be the bases of ${\rm Lie}\,G$ dual with respect to the Killing form $\Phi$.\;Identify ${\rm Lie}\,T$ with its dual space by means of $\Phi$.\;Let $\sigma$ be the sum of all positive roots.\;For every $\lambda\in \XT+$, put
\begin{equation}\label{cl}
c_\lambda:=\Phi(\lambda + \sigma, \lambda)+\Phi(\lambda^* + \sigma, \lambda^*)
\end{equation}
and consider on the space ${\mathcal S}^+(\lambda)\otimes_k {\mathcal S}^-(\lambda^*)$
the linear operator
\begin{equation}\label{Delta}
\Delta:=\displaystyle \sum_{i=1}^{n}(x_i\otimes x_i^*+x_i^*\otimes x_i).
\end{equation}
\begin{proposition} The following properties of an element $t\in {\mathcal S}^+(\lambda)\otimes_k {\mathcal S}^-(\lambda^*)$ are equivalent:
\begin{enumerate}[\hskip 4.2mm \rm(i)]
\item $t=s_\lambda$;
\item $\Delta (t)=-c_\lambda t$ and $t(e,e)=1$.
\end{enumerate}
\end{proposition}
\begin{proof} By \cite[Chap.\,VIII, \S6, Sect.\,4, Cor.]{Bou75}, the Casimir element $\Omega:=\sum_{i=1}^nx_ix_i^*\in {\mathscr U}$ acts on any simple ${\mathscr U}$-module with the highest weight $\gamma$ as scalar multiplication by
$\Phi(\gamma+\sigma, \gamma)$.\;Since $\Phi(\gamma+\sigma, \gamma)>0$ if $\gamma\neq 0$, the kernel of $\Omega$ in any finite dimensional ${\mathscr U}$-module $V$ coincides with $V^G$.\;We apply this to $V={\mathcal S}^+(\lambda)\otimes_k {\mathcal S}^-(\lambda^*)$.\;For any elements $f\in {\mathcal S}^+(\lambda)$, $h\in {\mathcal S}^-(\lambda^*)$, we deduce from \eqref{cl}, \eqref{Delta} the following:
\begin{equation*}
\begin{split}
\Omega(f\!\otimes\! h)\!&=\!\displaystyle  \sum_{i=1}^n\!\big(x_ix_i^*(f)\!\otimes\! h\!+\!x_i^*(f)\!\otimes\! x_i(h)\!
+\!x_i(f)\!\otimes\! x_i^*(h)\!+\!f\!\otimes\! x_ix_i^*(h)\!\big)\\
&=\Omega(f)\!\otimes\! h\!+\! f\!\otimes\! \Omega(h)\!+\!\Delta(f\!\otimes\! h)=c_\lambda (f\!\otimes\!h)\!+\!\Delta(f\!\otimes\! h).
\end{split}
\end{equation*}
Now Corollary \ref{corrr} and the aforesaid about ${\rm ker}\,\Omega$ complete the proof.
\end{proof}
\end{remark}

\begin{theorem}\label{iG}
Let $\lambda_1,\ldots, \lambda_m$ be a system of generators of the monoid
$\XT+$. Then $({\rm ker}\,\mu)^G$ is the linear span over $k$ of all monomials of the form
\begin{gather*}
    (s_{\lambda_1}-1)^{d_1}\cdots (s_{\lambda_m}-1)^{d_m}\!,\;\mbox{where $d_i\in {\mathbb N}$,  $d_1+\cdots+d_m>0$,}
\end{gather*}
where $s_{\lambda_i}$ is defined in
Corollary {\rm \ref{corrr}}.
\end{theorem}
\begin{proof}
By Theorem \ref{vect}(i), the linear span $L$ referred to in Theorem \ref{iG} is contained in $({\rm ker}\,\mu)^G$.\;In view of Theorem \ref{vect}, to prove the converse inclusion $({\rm ker}\,\mu)^G\subseteq L$, we have to show that, for every function
\begin{equation}\label{fff}
f\in ({\mathcal S}^+(\lambda)\otimes_k {\mathcal S}^-(\lambda^*)\!)^G,
\end{equation}
we have $f-f(e,e)\in L$.\;Since $\lambda_1,\ldots, \lambda_m$ is a system of generators of
$\XT+$, there are the integers $d_1,\ldots, d_m\in {\mathbb N}$ such that $\lambda=\sum_{i=1}^md_i\lambda_i$.\;From \eqref{s} and \eqref{grS+-} we then infer
that $h:=\prod_{i=1}^ms_{\lambda_i}^{d_i}
\in ({\mathcal S}^+(\lambda)\otimes_k {\mathcal S}^-(\lambda^*)\!)^G$ and $h(e,e)=1$.\;This, \eqref{fff}, and \eqref{10} imply that $f=f(e,e)h$.\;Therefore,
\begin{equation}\label{ef}
f-f(e,e)=f(e,e)(h-1)=f(e,e)\bigg(\!\displaystyle \prod_{i=1}^m((s_{\lambda_i}-1)+1)^{d_i}-1\!\bigg).
\end{equation}
The right-hand side of \eqref{ef} clearly lies in $L$. This completes the proof.
\end{proof}

\begin{theorem}\label{gi}
Let $\lambda_1,\ldots, \lambda_m$ be a system of generators of the monoid
$\XT+$.\;Then the ideal ${\rm ker}\,\mu$ in ${\mathcal S}^+\otimes_k {\mathcal S}^-$
is generated by $s_{\lambda_1}-1,\ldots, s_{\lambda_m}-1$, where $s_{\lambda_i}$ is defined in
Corollary {\rm \ref{corrr}}.
\end{theorem}
\begin{proof} This follows from Theorems \ref{ker} and \ref{iG}.
\end{proof}

\section{Presentation of ${\mathcal S}^\pm$}

If the group $G$ is semisimple, then the semigroup $\XT+$ has no units other than $0$.\;Hence
the set ${\mathscr H}$ of all indecomposable elements of $\XT+$ is finite,
\begin{equation}\label{Hilb}
{\mathscr H}=\{\lambda_1,\ldots, \lambda_d\},
\end{equation}
generates $\XT+$, and
every generating set of $\XT+$ contains ${\mathscr H}$ (see, e.g., \cite[Lemma 3.4.3]{Lo05}).\;Note that
${\mathscr H}$, called the {\it Hilbert basis of $\;\XT+$}, in general  is not a free generating system of $\XT+$ (i.e., it is not true that every element
$\alpha\in \XT+$ may be uniquely expressed in the form
$\alpha=\sum_{i=1}^d c_i\lambda_i$, $c_i\in{\mathbb N}$).\;Namely, it is free if and only if
$G=G_1\times\cdots\times G_s$ where
every $G_i$ is either a simply connected simple
algebraic group or isomorphic to ${\rm SO}_{n_i}$ for an
odd $n_i$ (see \cite[\S3]{St75}, \cite[Prop.\;4.1]{Ri79}, \cite[Prop.\;13.3]{Ri82}, \cite[Remark\;3.16]{Po11}).\;In
particular,
 if $G$ is simply connected, then ${\mathscr H}$
 coincides with the set of all fundamental weights
 and generates $\XT+$ freely.\;Note that $\lambda^*_i\in \mathscr H$ for every $i$.

To understand presentation of ${\mathcal S}^\pm$, denote respectively by ${\rm Sym}\,{\mathcal S}^\pm(\lambda_i)$ and ${\rm Sym}^m\,{\mathcal S}^\pm(\lambda_i)$  the symmetric algebra and the $m$th symmetric power of ${\mathcal S}^\pm(\lambda_i)$.\;The naturally ${\mathbb N}^d$-graded free commutative $k$-algebra
\begin{equation}\label{falga}
{\mathcal F}^\pm:={\rm Sym}\,{\mathcal S}^\pm(\lambda_1)\otimes_k\cdots\otimes_k {\rm Sym}\,{\mathcal S}^\pm(\lambda_d)
 \end{equation}
 may be viewed as the algebra of regular functions $k[L^\pm]$ on the vector space
 \begin{equation*}
 L^\pm:={\mathcal S}^\pm(\lambda_1)^*\oplus\cdots \oplus {\mathcal S}^\pm(\lambda_d)^*.
 \end{equation*}
 Let $e_{i}$ be the $i$th unit vector of ${\mathbb N}^d$ and let ${\mathcal F}^\pm_{p,q}$ be the homogeneous component of ${\mathcal F}^\pm$ of degree $e_p+e_q$.\;We have the natural isomorphisms of $G$-modules
\begin{equation}\label{ssss}
\varphi_{p,q}^\pm\colon {\mathcal F}^\pm_{p,q}\xrightarrow{\cong}{\mathcal S}^\pm_{p,q}:=\begin{cases}
{\mathcal S}^\pm(\lambda_p)\otimes_{k}{\mathcal S}^\pm(\lambda_q)
&\mbox{if $p\neq q$},\\
{\rm Sym}^{2}\,{\mathcal S}^\pm(\lambda_p) &\mbox{if $p=q$.}\end{cases}
\end{equation}

By Theorems  \ref{RR}, \ref{fg} the natural multiplication homomorphisms
\begin{equation}\label{phipsi}
\phi^\pm\colon {\mathcal F}^\pm\to {\mathcal S}^\pm\quad\mbox{and}\quad
\psi_{p,q}^\pm\colon {\mathcal S}^\pm_{p,q}\to {\mathcal S}^\pm(\lambda_p+\lambda_q)
\end{equation}
 are surjective.\;Since ${\mathcal F}^\pm$ is a polynomial algebra, the surjectivity of $\phi^\pm$ reduces finding a presentation of ${\mathcal S}^\pm$ by generators and relations to describing ${\rm ker}\,\phi^\pm$.\;If
 $d=\dim T$, the following explicit description of ${\rm ker}\,\phi^\pm$ is available:

\begin{theorem}\label{grS} Let $G$ be a connected semisimple group such that the Hilbert basis
{\rm \eqref{Hilb}} freely generates the semigroup $\,\XT+$.\;Then
\begin{enumerate}[\hskip 2.2mm\rm(i)]
\item the ideal ${\rm ker}\,\phi^\pm$ of the ${\mathbb N}^d$-graded
$k$-algebra ${\mathcal F}^\pm$ is homogeneous;
\item this ideal is generated by the union of all its homogeneous components of the total degree $2$;
 \item the set of these homogeneous components coincides with the set of all
subspaces
$\big(\varphi_{p,q}^\pm\big)^{-1}\big({\rm ker}\,\psi_{p,q}^\pm\big)$,
$1\leqslant p\leqslant q\leqslant d$.
\end{enumerate}
\end{theorem}
\begin{proof} This is the main result of \cite{KR87}.
\end{proof}

\begin{remark}
 In characteristic $0$ for the first time the proof of Theorem \ref{grS}
 was
 obtained (but not published) by B. Kostant; his proof appeared in \cite[Thm.\;1.1]{LT79}.\;In this case, \eqref{ssss} and the surjectivity of $\psi_{p,q}^\pm$ yield
that $\psi_{p,q}^\pm$ is the projection of ${\mathcal S}^\pm_{p,q}$ to the Cartan component
of ${\mathcal S}^\pm_{p,q}$, and ${\rm ker}\,\psi_{p,q}^\pm$ is the unique $G$-stable direct complement to
this component.\;The subspace ${\rm ker}\,\psi_{p,q}^\pm$
admits the following description using the notation of Remark \ref{000}
(loc.\,cit.).\;Let
$\{x_1,\ldots, x_n\}$ and $\{x_1^*,\ldots, x_n^*\}$ be the dual bases of
${\rm Lie}\,G$
with respect to
$\Phi$.\;Then
${\rm ker}\,\psi_{p,q}^+$ is the ima\-ge of the linear transformation $\big(\!\sum _{s=1}^n(x_s\otimes x_s^*+x_s^*\otimes x_s)\!\big)-2\Phi(\lambda_p^*, \lambda_q^*){\rm id}$
of the vector space\;${\mathcal S}_{p,q}^\pm$.
\end{remark}

Summing up, if
$G$ is a connected semisimple group such that the Hilbert basis
{\rm \eqref{Hilb}} freely generates the semigroup $\,\XT+$, then
the sought-for cano\-ni\-cal presentation of $k[G]$ is given
by the surjective
homomorphism
\begin{equation}\label{phiphi}
\phi:=\phi^+\otimes \phi^-\colon {\mathcal F}:=\mathcal F^+\otimes_k \mathcal F^-
\to k[G]
\end{equation}
 of the polynomial $k$-algebra ${\mathcal F}$ and the following generating system $\mathscr R$ of the ideal
 $\ker\,\phi$.\;Identify
 ${\mathcal F}^+$ and ${\mathcal F}^-$ with subalgebras of $\mathcal F$
 in the natural way.\;Then
 ${\mathscr R}={\mathscr R}_1 \bigsqcup {\mathscr R}_2$, where
 \begin{equation}\label{R1}
 {\mathscr R}_1=\displaystyle
 \bigcup_{p, q}
 \big(\big(\varphi_{p,q}^+\big)^{-1}\big({\rm ker}\,\psi_{p,q}^+\big)\,\mbox{\Large $\cup$}\, \big(\varphi_{p,q}^-\big)^{-1}\big({\rm ker}\,\psi_{p,q}^-\big)\big)
 \end{equation}
(see the definition of $\varphi_{p,q}^\pm$, $\psi_{p,q}^\pm$ in  \eqref{ssss}, \eqref{phipsi})
and
 \begin{equation}\label{R2}
 {\mathscr R}_2=\{s_{\lambda_1}-1,\ldots, s_{\lambda_d}-1\}
   \end{equation}
   (see the definition of $s_{\lambda_i}$ in Corollary \ref{corrr}).\;The elements of
 ${\mathscr R}_1$ (respectively, ${\mathscr R}_2$) are the Pl\"ucker-type (respectively, the ${\rm SL}_2$-type) relations of the presentation.

The canonical presentation of $k[G]$
is redundant.\;To
reduce the size of ${\mathscr R}_1$, we may replace every space $\ker\,\psi^\pm_{p,q}$ in \eqref{R1} by a basis of
this space.\;Finding such a basis falls
within  the framework of Standard Monomial Theory.

\section{An example}\label{example}

As an illustration, here we explicitly describe  the canonical presentation of $k[G]$
for $G={\rm SL}_n$, $n\geqslant 2$, and ${\rm char}\,k=0$.

Let $T$ be the maximal torus of diagonal matrices in $G$, and let $B^+$ (respectively, $B^-$) be the Borel subgroup of lower (respectively, upper) triangular matrices in $G$.\;Then
\begin{gather*}
\mathscr H\!=\!\{\varpi_1,\ldots, \varpi_{n-1}\},\;\;\mbox{where}\\[-.5mm]
\varpi_d\colon T\!\to\! k,\;\; {\rm diag}(a_1,\ldots, a_n)\mapsto a_{n-d+1}\cdots a_n.
\end{gather*}

Every pair $i_1, i_2\in [n]$
determines the function
\begin{equation}\label{fxij}
x_{i_1, i_2}\colon G\to k, \quad \begin{pmatrix}a_{1,1}&\ldots&a_{1,n}\\
\hdotsfor[2]{3}\\
a_{n,1}&\ldots&a_{n,n}
\end{pmatrix}
\mapsto a_{i_1,i_2}.
\end{equation}
The $k$-algebra generated by all functions
\eqref{fxij} is $k[G]$.

For every $d\in [n-1]$ and every sequence
$i_1,\ldots, i_d$ of $d$ elements of $[n]$, put
\begin{equation*}
f^-_{i_1,\ldots, i_d}:=\det\! \begin{pmatrix}x_{i_1,1}&\ldots&x_{i_1,d}\\
\hdotsfor[2]{3}\\
x_{i_d,1}&\ldots&x_{i_d,d}
\end{pmatrix}\!\!, \quad
f^+_{i_1,\ldots, i_d}:=\det \!\begin{pmatrix}x_{i_1,n-d+1}&\ldots&x_{i_1,n}\\
\hdotsfor[2]{3}\\
x_{i_d,n-d+1}&\ldots&x_{i_d,n}
\end{pmatrix}\!\!.
\end{equation*}

 For every fixed $d$, all functions $f^-_{i_1,\ldots, i_d}$ (respectively, $f^+_{i_1,\ldots, i_d}$) such that
 $i_1<\cdots< i_d$  are linearly independent over $k$ and their linear span over $k$ is
 the simple $G$-module ${\mathcal S}^-(\varpi_d)$ (respectively,  ${\mathcal S}^+(\varpi_d)$); see, e.g., \cite[Prop.\;3.2]{FT92}.\;Therefore, denoting by $x^\pm_{i_1,\ldots, i_d}$ the element $f^\pm_{i_1,\ldots, i_d}$ of the
 $k$-al\-geb\-ra ${\mathcal F}^\pm$ defined by  \eqref{falga}, we
identify ${\mathcal F}^\pm$ with  the polynomial $k$-algebra in variables $x^\pm_{i_1,\ldots, i_{d}}$, where $d$ runs
over $[n-1]$ and $i_1,\ldots, i_{d}$ runs over
$[n]_d$.\;Correspondingly, the $k$-algebra ${\mathcal F}$ is identified with
the polynomial $k$-algebra in the variables
$x^-_{i_1,\ldots, i_{d}}$ and $x^+_{i_1,\ldots, i_{d}}$,
the homomorphism \eqref{phiphi} takes the form
\begin{equation*}\label{phiphiphi}
\phi\colon {\mathcal F}
\to k[G],\quad x^+_{i_1,\ldots, i_{d}}\mapsto f^+_{i_1,\ldots, i_{d}},\;\; x^-_{i_1,\ldots, i_{d}}\mapsto f^-_{i_1,\ldots, i_{d}},
\end{equation*}
and $\phi^\pm=\phi|_{{\mathcal F}^\pm}$.\;Below the sets \eqref{R1} and \eqref{R2} are explicitly specified using this notation.

First, we will specify the Pl\"ucker-type relations.\;It
  is convenient to introduce the following elements of ${\mathcal F}^\pm$.\;Let  $i_1,\ldots, i_d$ be a sequence of $d\!\in\! [n\!-\!1]$ elements of $[n]$,  and let $j_1,\ldots, j_d$ be
the nondecreasing sequence obtained from $i_1,\ldots, i_d$ by permutation.\;Then we put
\begin{equation*}
x^\pm_{i_1,\ldots, i_d}=\begin{cases}
{\rm sgn}(i_1,\ldots, i_d)x^\pm_{j_1,\ldots, j_d}&\mbox{if $i_p\neq i_q$ for all $p\neq q$},\\
0& \mbox{otherwise}.
\end{cases}
\end{equation*}

The $k$-algebra $\mathcal S^\pm$ is the coordinate algebra of the affine multicone over the flag variety, see \cite{To79}.\;By the well-known classical Hodge's result \cite{Ho42}, \cite{Ho43} (see also \cite[p.\;434, Cor.\,1]{To79}) the ideal ${\rm ker}\,\phi^\pm$ is
  generated by all elements of the form
\begin{equation}\label{Plucker}
\displaystyle \sum_{l=1}^{q+1}(-1)^lx^\pm_{i_1,\ldots, i_{p-1}, j_l} x^\pm_{\hskip -.7mm j_1,\ldots, \widehat{j_l},\ldots, j_{q+1}},
\end{equation}
where $p$ and $q$ run over $[n-1]$, $p\leqslant q$, and $i_1,\ldots, i_{p-1}$ and $j_1,\ldots, j_{q+1}$ run over
$[n]_{p-1}$ and $[n]_{q+1}$ respectively.\;Since
every element \eqref{Plucker}
is homogeneous  of degree 2, this result together with
Theorem \ref{grS} imply that, for every fixed $p, q\in [n-1]$, the
set $\big(\varphi_{p,q}^\pm\big)^{-1}\big({\rm ker}\,\psi_{p,q}^\pm\big)$ in \eqref{R1} is the linear span of
all elements \eqref{Plucker}, where $i_1,\ldots, i_{p-1}$ and $j_1,\ldots, j_{q+1}$ run over
$[n]_{p-1}$ and $[n]_{q+1}$ respectively.\;This describes the Pl\"ucker-type relations \eqref{R1}.

Secondly,  we will describe $s_{\varpi_d}$.\;If
${\boldsymbol {i}}\in [n]_{n-d}$ is a sequence $i_1,\ldots, i_{n-d}$, we put $x^\pm_{\boldsymbol i}:=
x^\pm_{i_1,\ldots, i_{n-d}}$ and denote by
 ${\boldsymbol {i}}^*\in [n]_{d}$ the unique
 sequence $j_1,\ldots, j_{d}$
whose intersection with $i_1,\ldots, i_{n-d}$
is empty.\;Let ${\rm sgn}({\boldsymbol {i}}, {\boldsymbol {i}}^*)$ be the sign of the permutation
$(i_1,\ldots, i_{n-d},j_1,\ldots, j_{d})$.\;Then by \cite[Thm.\;3.1(b)]{FT92},
\begin{equation*}
s_{\varpi_d}=\displaystyle \sum_{\boldsymbol {i}\in [n]_{n-d}} {\rm sgn}({\boldsymbol {i}}, {\boldsymbol {i}}^*)x^-_{\boldsymbol {i}}x^+_{\boldsymbol {i}^*}.
\end{equation*}
This describes the ${\rm SL}_2$-type relations \eqref{R2}.

A similar description of the presentation of $k[G]$ may be given for
the classical groups $G$
 of several other types:\;for them, the Pl\"ucker-type (respectively, the ${\rm SL}_2$-type) relations are obtained
 using \cite{LT79}, \cite{LT85} (respectively,
\cite{FT92}).

\end{document}